\documentclass{article}
\usepackage{amsmath,amssymb,amsthm}
\usepackage{graphicx,subfigure,float,url,color}
\usepackage{enumitem}
\usepackage{mathrsfs}
\usepackage[colorlinks=true]{hyperref}
\usepackage{pdfsync}

\def\R{\mathrm{I\kern-0.21emR}}
\def\N{\mathrm{I\kern-0.21emN}}

\newcommand{\C} {\mathbb{C}}

\newcommand{\diag}{\operatorname{diag}}
\newcommand{\rank}{\operatorname{rank}}

\newcommand{\Rea}{\operatorname{Re}}

\newcommand{\cX}{\mathcal{X}}
\newcommand{\cU}{\mathcal{U}}
\newcommand{\cE}{\mathcal{E}}
\newcommand{\cB}{\mathcal{B}}

\newcommand{\cT}{\mathcal{T}}

\renewcommand{\geq}{\geqslant}
\renewcommand{\leq}{\leqslant}

\newtheorem{theorem}{Theorem}[section]
\newtheorem{proposition}[theorem]{Proposition}
\newtheorem{corollary}[theorem]{Corollary}
\newtheorem{lemma}[theorem]{Lemma}
\theoremstyle{definition}
\newtheorem{definition}[theorem]{Definition}
\newtheorem{remark}[theorem]{Remark}

\title{A flatness proof of the exponential turnpike phenomenon for linear-quadratic optimal control problems}

\author{
Michel Fliess\thanks{Sorbonne Universit\'e, Universit\'e Paris Cit\'e, CNRS, Inria, Laboratoire Jacques-Louis Lions, F-75005 Paris, France.}
\and
Claude Lobry
\and
Emmanuel Tr\'elat\thanks{Sorbonne Universit\'e, Universit\'e Paris Cit\'e, CNRS, Inria, Laboratoire Jacques-Louis Lions, F-75005 Paris, France.}
}

\date{}

\begin{document}

\maketitle

\begin{abstract}
We revisit finite-dimensional linear-quadratic optimal control from the viewpoint of differential flatness. If the pair $(A,B)$ is controllable, then the linear control system is flat, and every trajectory can be parametrized by a flat output and finitely many of its derivatives. Once this parametrization is inserted into the quadratic functional, the Euler-Lagrange condition becomes a linear differential equation with constant coefficients, or more generally a polynomial matrix differential equation. After reduction to Smith normal form, this equation decouples into scalar constant-coefficient equations, and its solutions are exponential-polynomials. This yields a viewpoint on the turnpike phenomenon that is quite different from the classical Hamiltonian-Riccati analysis: the turnpike mechanism appears directly from the stable-unstable splitting of the reduced flat equation. In particular, when the reduced Euler-Lagrange operator has no purely imaginary characteristic roots and when the endpoint constraints act nondegenerately on the stable and unstable modes, the optimal trajectory consists of a left-boundary layer, a right-boundary layer, and a long interior arc exponentially close to the static optimum. The same viewpoint also clarifies what changes when some weights are only semidefinite: the order of the reduced equation may drop, some endpoint conditions may become incompatible, and polynomial or oscillatory modes may destroy the exponential turnpike. It also gives a natural meaning to certain endpoint constraints on the control and on finitely many derivatives of the control: such traces are not defined on the ambient $L^2$ control space, but they are meaningful on the smooth extremals selected by the reduced Euler-Lagrange equation. We formulate this principle as a general theorem and illustrate it in detail on the double integrator.
\end{abstract}

\section{Introduction}
\subsection{Main objective}
The exponential turnpike phenomenon is a central structural property of long-horizon optimal control problems. Roughly speaking, it states that for large final time $T$, an optimal trajectory spends most of its time exponentially close to the optimal solution of the associated static optimization problem, except for two transient layers near $t = 0$ and $t = T$. In the finite-dimensional linear-quadratic setting, the classical proof relies on the Pontryagin maximum principle, the Hamiltonian matrix, and Riccati theory (see in particular \cite{TrelatZuazua2015,TrelatZhangZuazua2018,GruneGuglielmi2021}).

The purpose of the present article is to present another approach, of a rather different nature, based on flatness. The starting observation is very simple. Assume that the pair $(A,B)$ is controllable. Then the linear control system
\begin{equation}
\label{probintro}
\dot x(t) = A x(t) + B u(t)
\end{equation}
is flat; equivalently, its system module is free over the polynomial ring $\R[D]$, where $D = \frac{\mathrm{d}}{\mathrm{d}t}$ (see \cite{Fliess1990,FliessLevineMartinRouchon1995,JoinDelaleauFliess2025}). Hence, one may choose a flat output $y$ and write
\begin{equation*}
x = \cX(D) y, \qquad u = \cU(D) y,
\end{equation*}
with polynomial matrix differential operators $\cX(D)$ and $\cU(D)$. For a linear-quadratic functional, this transforms the optimal control problem into a higher-order variational problem in the flat output. The corresponding Euler-Lagrange equation is a constant-coefficient differential equation,
\begin{equation}
\label{eulerintro}
\cE(D) y = 0,
\end{equation}
where $\cE(D)$ is a self-adjoint polynomial matrix. Since $\R[D]$ is a principal ideal domain, one may reduce $\cE(D)$ to Smith normal form; that is, diagonalize it by left and right multiplication by unimodular polynomial matrices (see, e.g., \cite{Kailath1980,GohbergLancasterRodman1982}). The equation \eqref{eulerintro} is then reduced to scalar constant-coefficient equations. Their solutions are finite sums of exponential-polynomials. This is precisely the point at which the turnpike mechanism becomes transparent.

Indeed, if the characteristic roots of the reduced equation all have a nonzero real part, then every solution splits into a stable part, exponentially small away from $t = 0$, and an unstable part, exponentially small away from $t = T$. If, in addition, the two-point boundary conditions constrain the stable directions at the left end and the unstable directions at the right end in a nondegenerate way, then one obtains an exponential turnpike estimate. Conversely, if purely imaginary roots are present, then oscillatory or polynomial modes survive, and the exponential turnpike may fail. In semidefinite linear-quadratic problems, this is exactly what happens: some weights vanish, the reduced equation may lose order, and the admissible endpoint conditions have to be adjusted to the effective order of the flat equation.

This flat reduction provides a genuinely different viewpoint on the turnpike phenomenon. Instead of starting from the Hamiltonian matrix or from Riccati theory as in \cite{TrelatZuazua2015}, one reads the long-time behavior directly on the characteristic roots of the reduced Euler-Lagrange operator. In the regular case, this recovers, in flat coordinates, the same hyperbolic mechanism as in the classical theory. In semidefinite cases, however, the flat equation makes visible phenomena that are less transparent in the Hamiltonian formulation: order drop, loss of admissible boundary conditions, and the possible appearance of polynomial or oscillatory modes. In that sense, the flatness viewpoint does not merely rephrase the classical proof: it isolates the structural mechanism behind exponential turnpikes and extends naturally to regimes in which the control weight is no longer coercive.

The main objective of this article is to make the above idea precise. Our main theorem isolates two separate ingredients:
\begin{itemize}
\item A \emph{hyperbolicity condition} on the reduced Euler-Lagrange operator, namely the absence of purely imaginary characteristic roots.
\item An \emph{admissibility condition} on the endpoint constraints, meaning that the corresponding two-point boundary value problem constrains the stable modes at the left end and the unstable modes at the right end in a nondegenerate way.
\end{itemize}
Under these two assumptions, the exponential turnpike follows from a direct argument. In the regular case $R>0$ (where $R$ denotes the control-weight matrix in the running cost introduced later in \eqref{mainprob}), this gives an alternative reading of the classical Hamiltonian proof. In semidefinite cases, it clarifies what remains true and what breaks down. To the best of our knowledge, the deterministic finite-dimensional exponential turnpike literature for linear-quadratic problems has mainly focused on regular situations in which the control weight is coercive or, equivalently, on settings in which the optimality system is treated through the classical Hamiltonian hyperbolicity framework (see, e.g., \cite{TrelatZuazua2015,TrelatZhangZuazua2018,GruneGuglielmi2021,GuglielmiLi2024,TrelatZuazua2025}). One of the points of the present article is to show that the flat reduction still makes sense when $R \geq 0$ is only semidefinite, and that it then reveals, in a very explicit way, the mechanisms of order drop and boundary incompatibility. A related point, emphasized later in Remark~\ref{controltracesremark}, is that the same flat reduction provides a natural framework for imposing endpoint constraints on $u$ and on finitely many derivatives of $u$. Although such traces do not exist on the whole $L^2$ admissible class, they become meaningful on the smooth extremals of the reduced Euler-Lagrange equation and are governed by the same admissibility mechanism.

\subsection{Bibliographical comments}
The flatness viewpoint used here lies at the intersection of two bodies of literature. The bibliography cited below is selective rather than exhaustive. It is organized to document the three main ingredients of the paper: flatness, turnpike theory, and singular or relaxed variants.

\paragraph{Turnpike theory.}
The modern control-theoretic proof of exponential turnpikes in finite dimension is due to \cite{TrelatZuazua2015}, see also the Hilbert-space extension \cite{TrelatZhangZuazua2018}, the necessary and sufficient conditions for continuous-time linear-quadratic problems in \cite{GruneGuglielmi2021}, the necessary conditions in generalized linear-quadratic settings obtained in \cite{GuglielmiLi2024}, the value-function and Hamilton-Jacobi viewpoint developed in \cite{EsteveKouhkouhPighinZuazua2022}, the two-term asymptotic expansion obtained in \cite{AskovicTrelatZidani2024}, and the recent survey \cite{TrelatZuazua2025}. The present article does not compete with the Hamiltonian approach. Rather, it identifies a very transparent and sufficient mechanism behind it: once the problem is written in flat variables, the turnpike follows from the stable-unstable splitting of a constant-coefficient higher-order equation.

\paragraph{Flatness and module theory.}
For linear systems, the equivalence between controllability and flatness, or equivalently the freeness of the system module, goes back to the algebraic viewpoint developed in \cite{Fliess1990,FliessLevineMartinRouchon1995}. Brunovsk\'y's canonical form \cite{Brunovsky1970} is today the most popular for viewing the flat parametrization, while the monograph \cite{Levine2009} provides a systematic presentation of flatness-based control. The recent papers \cite{JoinDelaleauFliess2024,JoinDelaleauFliess2025} emphasize the aspect used here: after flat reduction, the Euler-Lagrange equation of a linear-quadratic problem becomes a constant-coefficient differential equation, which is especially convenient for open-loop constructions and boundary-value problems \cite{batna,JoinDelaleauFliess2025}.

\paragraph{Semidefinite and singular regimes.}
The flat reduction is particularly useful when some weights vanish. Then the reduced equation may lose order, which immediately reveals why some endpoint conditions become incompatible. This connects the present work with cheap control and singular control (see \cite{KwakernaakSivan1972,OReilly1983,Goh1966,SaberiSannuti1987,KokotovicKhalilOReilly1986}). If one enlarges the admissible class to measures or relaxed controls, one recovers the possibility of boundary impulses; standard references are \cite{Vinter2000,VinterLewis1978}. The double-integrator example of Section~\ref{double} should be read in this light.

\paragraph{Linear rather than exponential turnpikes.}
When a zero root appears in the reduced Euler-Lagrange operator, polynomial modes enter the picture. In such cases, the correct asymptotic behavior is often a linear turnpike rather than an exponential one. This is the situation analyzed in \cite{Trelat2023}. The flatness viewpoint makes this mechanism immediately visible at the level of the characteristic polynomial.

\paragraph{Operational calculus and exponential-polynomials.}
The decomposition of solutions into exponential-polynomials can be read either from ordinary differential equations with constant coefficients or from operational calculus, as in \cite{Schwartz1965,Yosida1984}. In this paper, we will keep the proof at the level of ordinary differential equations and polynomial matrices, but the operational viewpoint is perfectly consistent with the present argument and was one of its motivations.

\subsection{Structure of the article}
The article is organized as follows. Section~\ref{double} studies in detail the double integrator, which already contains all the relevant mechanisms. Section~\ref{flatsec} formulates the flat reduction for general controllable linear systems and identifies the reduced Euler-Lagrange operator. Section~\ref{hypsec} proves an abstract hyperbolic estimate for linear two-point problems. Section~\ref{mainsec} combines these ingredients into a flatness-based exponential turnpike theorem. The final section contains the conclusion, perspectives, and bibliographical comments.

\section{A motivating example: the double integrator}
\label{double}

Consider the double integrator
\begin{equation}
\label{double-syst}
\dot x_1(t) = x_2(t), \qquad \dot x_2(t) = u(t),
\end{equation}
with the quadratic functional
\begin{equation}
\label{double-cost-raw}
J_T(x,u) = \frac{1}{2} \int_0^T \Big( q_1 \big( x_1(t) - \alpha_1 \big)^2 + q_2 \big( x_2(t) - \alpha_2 \big)^2 + r \big( u(t) - \beta \big)^2 \Big) \, \mathrm{d}t,
\end{equation}
where $q_1,q_2,r \geq 0$ and $T > 0$. The associated static optimization problem is
\begin{equation*}
\min \left\{ \frac{1}{2} \left( q_1 ( x_1 - \alpha_1 )^2 + q_2 ( x_2 - \alpha_2 )^2 + r ( u - \beta )^2 \right) \ \mid\ x_2 = 0,\ u = 0 \right\}.
\end{equation*}
Its unique minimizer is $(\bar x_1,\bar x_2,\bar u) = (\alpha_1,0,0)$ whenever $q_1 > 0$. Thus the turnpike candidate is the equilibrium $(\alpha_1,0,0)$.

The flat output is $y = x_1 - \alpha_1$. Then
\begin{equation}
\label{double-flat}
x_1 = \alpha_1 + y, \qquad x_2 = \dot y, \qquad
u = \ddot y.
\end{equation}
Substituting \eqref{double-flat} into \eqref{double-cost-raw} yields the higher-order Lagrangian
\begin{equation}
\label{double-flat-cost}
J_T(y) = \frac{1}{2} \int_0^T \left( q_1 y(t)^2 + q_2 ( \dot y(t) - \alpha_2 )^2 + r ( \ddot y(t) - \beta )^2 \right)  \mathrm{d}t.
\end{equation}
The corresponding stationarity condition is the classical higher-order Euler-Lagrange equation (see, e.g., \cite{Bliss1946,CourantHilbert1953,GelfandFomin1963,Hestenes1966}).
Hence every stationary point satisfies the Euler-Lagrange equation
\begin{equation*}
r y^{(4)}(t) - q_2 \ddot y(t) + q_1 y(t) = 0.
\end{equation*}
Observe that the constants $\alpha_2$ and $\beta$ disappear from the interior equation: they only contribute to the boundary terms of the variational formula. This elementary fact already illustrates a general principle used throughout the paper.

\subsection{The regular case}
Assume that the full state is prescribed at both ends,
\begin{equation}
\label{double-fullbc}
y(0) = y_0, \qquad \dot y(0) = v_0, \qquad y(T) = y_T, \qquad \dot y(T) = v_T.
\end{equation}
This corresponds to prescribing $(x_1(0),x_2(0),x_1(T),x_2(T))$.

\begin{proposition}
\label{doubleprop1}
Assume that $q_1 > 0$ and $r > 0$, with $q_2 \geq 0$. Then the characteristic polynomial
\begin{equation*}
p(\lambda) = r \lambda^4 - q_2 \lambda^2 + q_1
\end{equation*}
has no purely imaginary root. For all endpoint data \eqref{double-fullbc} and every $T > 0$, the variational problem \eqref{double-flat-cost} has a unique minimizer. Moreover there exist $T_0 > 0$ and $C,\mu > 0$, independent of $T$ and of $t \in [0,T]$, such that for every $T \geq T_0$
\begin{equation}
\label{double-est-regular}
\vert y(t) \vert + \vert \dot y(t) \vert + \vert \ddot y(t) \vert \leq C \left( e^{-\mu t} + e^{-\mu (T - t)} \right) \qquad\forall t\in[0,T].
\end{equation}
Consequently the optimal trajectory of \eqref{double-syst}-\eqref{double-cost-raw} satisfies the exponential turnpike estimate around $(\alpha_1,0,0)$.
\end{proposition}

\begin{proof}
If $\lambda = i \omega$ with $\omega \in \R$, then $p(i \omega) = r \omega^4 + q_2 \omega^2 + q_1 > 0$, thus there is no purely imaginary root. Every root therefore has nonzero real part. Since \eqref{double-flat-cost} is strictly convex on the affine space determined by \eqref{double-fullbc}, the minimizer is unique for every $T > 0$. Rewriting the stationary solution in the stable-unstable form described later in Theorem~\ref{abstract-bvp} yields \eqref{double-est-regular} for large $T$. The conclusion for $(x_1,x_2,u)$ follows from \eqref{double-flat}.
\end{proof}

\begin{remark}
\label{doubleclassical}
Proposition~\ref{doubleprop1} lies within the scope of the classical exponential turnpike theory for regular linear-quadratic problems (see \cite{TrelatZuazua2015}). In that case the flat proof given here should be viewed as an alternative reading of the usual Hamiltonian-Riccati argument: both rely on the same hyperbolic splitting, but the flat reduction makes it visible at the level of a scalar constant-coefficient equation.
\end{remark}

\subsection{Cheap control and order drop}
Assume now that $q_1 > 0$, $q_2 > 0$ and $r = 0$. This regime is classically called \emph{cheap control} (see \cite{KwakernaakSivan1972,OReilly1983,SaberiSannuti1987,KokotovicKhalilOReilly1986}). Indeed, the control variable is no longer penalized in the running cost, or equivalently, one may think of it as the singular limit of problems in which the control term is multiplied by a small parameter $\rho > 0$ and $\rho \to 0$. In that limit the control becomes ``cheap'' compared with the state cost, and one expects fast boundary corrections, high-gain effects, or even impulsive limits. Then the reduced equation becomes
\begin{equation}
\label{double-cheap-ode}
q_2 \ddot y(t) - q_1 y(t) = 0.
\end{equation}
Set $\omega = \sqrt{\frac{q_1}{q_2}}$. The general solution is
\begin{equation}
\label{double-cheap-sol}
y(t) = c_s e^{-\omega t} + c_u e^{-\omega (T - t)}.
\end{equation}
The order has dropped from $4$ to $2$. This is the decisive phenomenon: prescribing \eqref{double-fullbc} now gives four scalar conditions for a second-order equation and is therefore generically impossible.

\begin{remark}
\label{doublecompat}
The classical problem with prescribed full initial and final states is overdetermined when $r = 0$. A classical solution exists only if the four endpoint data satisfy two compatibility relations. Equivalently, if $y(0)$ and $\dot y(0)$ are prescribed, then necessarily
$$
\begin{pmatrix}
y(T) \\
\dot y(T)
\end{pmatrix}
=
\begin{pmatrix}
\cosh (\omega T) & \omega^{-1} \sinh (\omega T) \\
\omega \sinh (\omega T) & \cosh (\omega T)
\end{pmatrix}
\begin{pmatrix}
y(0) \\
\dot y(0)
\end{pmatrix}.
$$
This is why, in the cheap-control regime, the appropriate classical endpoint conditions no longer coincide with the full initial and final state constraints of the original system.
\end{remark}

The flat viewpoint shows immediately which two-point conditions are admissible.

\begin{proposition}
\label{doubleprop2}
Assume that $q_1 > 0$, $q_2 > 0$ and $r = 0$. Consider the two scalar boundary conditions
\begin{equation}
\label{double-mixedbc}
a_0 y(0) + b_0 \dot y(0) = c_0, \qquad a_T y(T) + b_T \dot y(T) = c_T.
\end{equation}
If
\begin{equation}
\label{double-admissibility}
a_0 - \omega b_0 \neq 0, \qquad a_T + \omega b_T \neq 0,
\end{equation}
then there exist $T_0 > 0$ and $C > 0$ such that, for every $T \geq T_0$, the problem \eqref{double-cheap-ode}-\eqref{double-mixedbc} has a unique solution and
\begin{equation}
\label{double-est-cheap}
\vert y(t) \vert + \vert \dot y(t) \vert \leq C \left( e^{-\omega t} + e^{-\omega (T - t)} \right)  \qquad \forall t \in [0,T].
\end{equation}
Thus the corresponding reduced optimal trajectory has an exponential turnpike.
\end{proposition}

\begin{proof}
Substituting \eqref{double-cheap-sol} into \eqref{double-mixedbc} gives
$$
\begin{pmatrix}
a_0 - \omega b_0 & e^{-\omega T} (a_0 + \omega b_0) \\
e^{-\omega T} (a_T - \omega b_T) & a_T + \omega b_T
\end{pmatrix}
\begin{pmatrix}
c_s \\
c_u
\end{pmatrix}
=
\begin{pmatrix}
c_0 \\
c_T
\end{pmatrix}.
$$
The determinant equals
$$
(a_0 - \omega b_0) (a_T + \omega b_T) + \mathrm{O} \big( e^{-2 \omega T} \big),
$$
which is nonzero under \eqref{double-admissibility}. Hence $(c_s,c_u)$ is uniquely determined and remains uniformly bounded when $(c_0,c_T)$ ranges in a bounded set. Estimate \eqref{double-est-cheap} follows from \eqref{double-cheap-sol}.
\end{proof}

\begin{remark}
The original problem with full initial and final state constraints may nevertheless admit a solution in a relaxed class if controls are allowed to be Radon measures. Then the reduced variable $v = x_2$ follows the classical optimal trajectory of the first-order problem, while the missing endpoint conditions on $x_2$ are enforced by Dirac masses at $t = 0$ and $t = T$. This is the familiar cheap-control or impulsive-control mechanism (see, e.g., \cite{VinterLewis1978} and \cite[Chapter 8]{Vinter2000}).
\end{remark}

\subsection{When hyperbolicity is lost}
What matters in the flat reduction is the spectral condition on the reduced Euler-Lagrange operator, namely hyperbolicity. The following cases illustrate what happens when this condition, or more precisely the absence of purely imaginary roots, fails.

\begin{proposition}
\label{doubleprop3}
For the double integrator, the reduced characteristic polynomial is $p(\lambda) = r \lambda^4 - q_2 \lambda^2 + q_1$. The following alternatives occur:
\begin{enumerate}[label={\rm (\roman*)}]
\item If $q_1 > 0$ and $r > 0$, all characteristic roots have nonzero real part and the exponential turnpike holds for the standard four endpoint conditions.
\item If $q_1 > 0$, $q_2 > 0$ and $r = 0$, the reduced equation remains hyperbolic, but its order drops to $2$. Only two scalar boundary conditions are admissible in the classical setting.
\item If $q_1 > 0$, $q_2 = 0$ and $r = 0$, the reduced equation is simply $q_1 y = 0$. Hence $y = 0$ is the only classical trajectory of the centered problem. Nontrivial endpoint data are incompatible.
\item If $q_1 = 0$, $q_2 > 0$ and $r > 0$, then $p(\lambda) = \lambda^2 ( r \lambda^2 - q_2 )$.
The root $\lambda = 0$ has multiplicity $2$. Polynomial modes appear, the exponential turnpike fails in general, and one enters the linear turnpike regime of \cite{Trelat2023}.
\item If $q_1 = 0$, $q_2 > 0$ and $r = 0$, then $\ddot y = 0$. Only affine trajectories remain, hence there is no exponential turnpike.
\item If $q_1 = 0$, $q_2 = 0$ and $r > 0$, then $y^{(4)} = 0$. The optimal trajectory is polynomial and there is again no exponential turnpike.
\end{enumerate}
\end{proposition}

\begin{proof}
The proof is immediate from the roots of $p(\lambda)$. The only point worth emphasizing is that the presence of the root $\lambda = 0$ creates polynomial modes. Such modes may be small near one endpoint and still grow linearly or polynomially in the interior, which is incompatible with an exponential turnpike estimate.
\end{proof}

The conclusions of Propositions~\ref{doubleprop1}-\ref{doubleprop3} already contain the general message of the paper, namely:
\begin{itemize}
\item The interior dynamics of the optimal solution is encoded by the reduced Euler-Lagrange equation in the flat output.
\item The relevant spectral condition is the absence of purely imaginary characteristic roots.
\item When some weights vanish, the order of that equation may decrease. Then the admissible endpoint conditions have to be matched to the reduced order.
\end{itemize}
We now formulate this in full generality.

\section{Flat reduction of controllable linear-quadratic problems}
\label{flatsec}

\subsection{The flat parametrization}
Consider the centered linear-quadratic optimal control problem
\begin{equation}
\label{mainprob}
\left\{
\begin{aligned}
& \dot x(t) = A x(t) + B u(t), \\
& M_0 x(0) + M_1 x(T) = \gamma, \\
& \min J_T(x,u) = \frac{1}{2} \int_0^T \left( x(t)^\top Q x(t) + u(t)^\top R u(t) \right) \mathrm{d}t,
\end{aligned}
\right.
\end{equation}
where $x(t) \in \R^n$, $u(t) \in \R^m$, the matrices $Q \in \R^{n \times n}$ and $R \in \R^{m \times m}$ are symmetric nonnegative, and $M_0,M_1 \in \R^{k \times n}$ satisfy the full-row-rank condition%
\footnote{This condition simply means that the $k$ scalar boundary constraints are independent. Typical examples are: prescribed initial state, obtained by $k = n$, $M_0 = I_n$, $M_1 = 0$; prescribed final state, obtained by $k = n$, $M_0 = 0$, $M_1 = I_n$; prescribed initial and final states, obtained by $k = 2n$, $M_0 = \begin{pmatrix} I_n \\ 0 \end{pmatrix}$ and $M_1 = \begin{pmatrix} 0 \\ I_n \end{pmatrix}$. More generally, one may impose any independent linear combination of the initial and final states.}
\begin{equation*}
\rank \big( M_0 \ \ M_1 \big) = k.
\end{equation*}
The centered formulation is sufficient for turnpike estimates around an arbitrary static optimum $(\bar x,\bar u)$, because one simply replaces $(x,u)$ with $(x - \bar x,u - \bar u)$; affine terms in the original running cost only contribute lower-order or boundary terms and do not modify the reduced homogeneous operator.

Assume throughout this section that the pair $(A,B)$ is controllable.

\begin{proposition}
\label{flatprop}
If $(A,B)$ is controllable, then the system \eqref{probintro} is flat. More precisely, there exist polynomial matrices
$$
\cX(D) \in \R[D]^{n \times m}, \qquad \cU(D) \in \R[D]^{m \times m},
$$
and a flat output $y : [0,T] \to \R^m$ such that every trajectory of \eqref{probintro} can be written as
\begin{equation}
\label{flatparam}
x = \cX(D) y, \qquad u = \cU(D) y.
\end{equation}
In Brunovsk\'y coordinates, if $\nu_1,\ldots,\nu_m$ are the controllability indices, then $\sum_{i=1}^m \nu_i = n$, the state depends on the derivatives of $y_i$ up to order $\nu_i - 1$, and the control depends on the derivatives of $y_i$ up to order $\nu_i$.
\end{proposition}

\begin{proof}
This is standard in flatness theory for controllable linear systems. One may either use the Brunovsk\'y normal form \cite{Brunovsky1970} or the module-theoretic characterization of flatness \cite{Fliess1990,FliessLevineMartinRouchon1995}. The last statement follows immediately from the Brunovsk\'y chains of integrators.
\end{proof}

\begin{remark}
Any two flat outputs are related by a unimodular polynomial transformation. Therefore, the spectral information extracted below from the reduced Euler-Lagrange operator is intrinsic, although its matrix expression depends on the chosen flat output.
\end{remark}

Substituting \eqref{flatparam} into \eqref{mainprob} gives a higher-order variational problem in $y$:
\begin{equation}
\label{flat-cost}
J_T(y) = \frac{1}{2} \int_0^T \Big( \big( \cX(D) y \big)^\top Q \big( \cX(D) y \big) + \big( \cU(D) y \big)^\top R \big( \cU(D) y \big) \Big) \, \mathrm{d}t,
\end{equation}
under the endpoint constraint
\begin{equation}
\label{flat-bc}
M_0 \cX(D) y(0) + M_1 \cX(D) y(T) = \gamma.
\end{equation}
The Lagrangian in \eqref{flat-cost} depends only on finitely many derivatives of $y$. By Proposition~\ref{flatprop}, the highest derivative of $y_i$ entering the integrand is at most $\nu_i$.

\subsection{The reduced Euler-Lagrange operator}
For a polynomial matrix differential operator $\mathcal{P}(D) = \sum_{j=0}^\ell P_j D^j$, its formal adjoint is defined by
\begin{equation}
\label{formaladj}
\mathcal{P}(D)^\ast = \sum_{j=0}^\ell (-D)^j P_j^\top = \mathcal{P}(-D)^\top.
\end{equation}

\begin{proposition}
\label{ELprop}
Defining the Euler-Lagrange operator by
\begin{equation}
\label{ELoperator}
\cE(D) = \cX(D)^\ast Q \cX(D) + \cU(D)^\ast R \cU(D),
\end{equation}
The Euler-Lagrange equation associated with \eqref{flat-cost} is
\begin{equation}
\label{ELmatrix}
\cE(D) y = 0.
\end{equation}
The Euler-Lagrange operator $\cE(D)$ is self-adjoint in the sense that $\cE(D)^\ast = \cE(D)$.
Moreover, after repeated integration by parts, the first variation of $J_T$ has the form
\begin{equation}
\label{variation-form}
\delta J_T(y) = \int_0^T \big\langle \cE(D) y(t),\delta y(t) \big\rangle \, \mathrm{d}t + \cT_T \big( Y(0),Y(T),\delta Y(0),\delta Y(T) \big),
\end{equation}
where $Y(t)$ denotes a finite jet of $y$ and $\cT_T$ is a boundary bilinear form. Hence a stationary point is characterized by \eqref{ELmatrix} together with the prescribed endpoint conditions \eqref{flat-bc} and the natural boundary conditions obtained by requiring the boundary term to vanish for all admissible variations.
\end{proposition}

\begin{proof}
Formula \eqref{ELmatrix} is the standard higher-order Euler-Lagrange equation for a Lagrangian depending on finitely many derivatives (see, e.g., \cite{Bliss1946, CourantHilbert1953, GelfandFomin1963, Hestenes1966}). Self-adjointness is immediate from \eqref{formaladj}. The boundary term in \eqref{variation-form} is obtained by integrating by parts all derivatives falling on $\delta y$.
\end{proof}

\begin{remark}
The reduced equation \eqref{ELmatrix} is the flat counterpart of the Hamiltonian extremal system. In the regular case $R > 0$, its effective order is $2n$. Indeed, in Brunovsk\'y coordinates, the principal part of $\cU(D)$ is diagonal with entries $D^{\nu_i}$; whence the principal symbol of $\cE(D)$ has a determinant of degree $2 \sum_{i=1}^m \nu_i = 2n$. If $R$ is only semidefinite, this order may drop. The double-integrator example of Section~\ref{double} shows how such an order drop changes the admissible endpoint constraints.
\end{remark}

The Euler-Lagrange operator $\cE(D)$ can then be scalarized by the Smith decomposition, that is, by diagonalizing the polynomial matrix $\cE(D)$ through left and right multiplication by unimodular polynomial matrices (see, e.g., \cite{Kailath1980,GohbergLancasterRodman1982}).

\begin{proposition}
\label{smithprop}
There exist unimodular polynomial matrices $U(D),V(D) \in \R[D]^{m \times m}$ and monic polynomials (invariant factors) $d_1(D),\ldots,d_m(D)$, with $d_1 \,\vert\, d_2 \,\vert\, \cdots \,\vert\, d_m$, such that
\begin{equation}
\label{smith}
U(D) \cE(D) V(D) = \diag \big( d_1(D),\ldots,d_m(D) \big).
\end{equation}
If $z = V(D)^{-1} y$, then \eqref{ELmatrix} is equivalent to the family of scalar equations
\begin{equation}
\label{smithscalar}
d_j(D) z_j = 0, \qquad j = 1,\ldots,m.
\end{equation}
Consequently, every component of every stationary flat output is a finite sum of terms of the form
\begin{equation*}
t^\ell e^{a t} \cos (\omega t), \qquad t^\ell e^{a t} \sin (\omega t),
\end{equation*}
where $a + i \omega$ runs over the roots of the invariant factors $d_j$ and $\ell$ is bounded by the multiplicity minus one.
\end{proposition}

\begin{proof}
Since $\R[D]$ is a principal ideal domain, the Smith decomposition exists. Formula \eqref{smithscalar} follows from \eqref{smith}. The description of solutions to \eqref{smithscalar} is the usual one for linear scalar ordinary differential equations with constant coefficients.
\end{proof}

\begin{definition}
We say that the reduced Euler-Lagrange operator $\cE(D)$ defined by \eqref{ELoperator} is \emph{hyperbolic} if
\begin{equation}
\label{hyper-cond}
\det \cE(i \omega) \neq 0
\qquad
\forall \omega \in \R.
\end{equation}
Equivalently, none of the invariant factors $d_j$ has a purely imaginary root.
\end{definition}

\begin{remark}
\label{hyperzero}
The condition \eqref{hyper-cond} must include $\omega = 0$. Indeed, the frequency $\omega = 0$ corresponds to the characteristic root $\lambda = 0$, which is itself purely imaginary. Allowing $\det \cE(0) = 0$ would leave room for zero roots and hence for polynomial modes, exactly as in Proposition~\ref{doubleprop3}(iv)-(vi). In regular linear-quadratic situations, there is no contradiction: for the double integrator, for instance, one has $\cE(0) = q_1 > 0$ whenever $q_1 > 0$.
\end{remark}

The following frequency-domain interpretation is particularly useful. 

\begin{lemma}
For every $\omega \in \R$ and every $\xi \in \C^m$,
\begin{equation}
\label{freqidentity}
\xi^\ast \cE(i \omega) \xi = \big\Vert Q^{\frac{1}{2}} \cX(i \omega) \xi \big\Vert^2 + \big\Vert R^{\frac{1}{2}} \cU(i \omega) \xi \big\Vert^2 \geq 0.
\end{equation}
Hence $\det \cE(i \omega) = 0$ if and only if there exists a nonzero harmonic mode $y(t) = \Rea \big( \xi e^{i \omega t} \big)$ whose induced state and control are invisible to the running cost.
\end{lemma}

\begin{proof}
Using \eqref{ELoperator} and the reality of the coefficients, one has
$$
\cE(i \omega) = \cX(i \omega)^\ast Q \cX(i \omega) + \cU(i \omega)^\ast R \cU(i \omega).
$$
Taking the Hermitian form against $\xi$ gives \eqref{freqidentity}. Since the right-hand side is a sum of nonnegative terms, it vanishes if and only if both terms vanish.
\end{proof}

\begin{remark}
\label{quartetremark}
Because $\cE(-\lambda)^\top = \cE(\lambda)$, the characteristic roots of $\cE$ are symmetric with respect to both the real axis and the origin: if $\lambda$ is a root, then $-\lambda$, $\bar \lambda$ and $-\bar \lambda$ are roots as well. This is the flat analogue of the classical Hamiltonian symmetry that was used in \cite[Lemmas 1 and 2]{TrelatZuazua2015}.
\end{remark}

\subsection{Boundary conditions and effective order}

Let $N$ denote the total order of the Smith form:
\begin{equation}
\label{Ndef}
N = \sum_{j=1}^m \deg d_j.
\end{equation}
This is the dimension of the solution space of \eqref{ELmatrix}. Under hyperbolicity, $N$ is even and half of the modes are stable while the other half are unstable (see Remark \ref{quartetremark}).

The prescribed initial and final state constraints \eqref{flat-bc} provide only part of the required endpoint information. The remaining conditions come from the natural transversality relations hidden in the boundary term of Proposition~\ref{ELprop}. Altogether, the stationary problem becomes a linear two-point boundary value problem with exactly $N$ scalar conditions. In the regular case $R > 0$, one has $N = 2n$, so fixed initial and final states provide exactly the right number of conditions. If $R$ is semidefinite, then $N$ may be strictly smaller than $2n$, and some of the original endpoint constraints become incompatible. The double integrator with $r = 0$ is the simplest manifestation of this fact.

\begin{remark}
\label{controltracesremark}
The same discussion suggests a less standard but natural extension of the linear-quadratic framework. Suppose that, in addition to endpoint conditions on the state, one wishes to impose linear conditions on $u(0)$, $u(T)$, or more generally on finitely many derivatives of $u$ at $t = 0$ and $t = T$. In the original formulation of \eqref{mainprob}, this is not meaningful on the whole admissible class, since controls are only required to belong to $L^2(0,T;\R^m)$. After flat reduction, however, every stationary trajectory satisfies the polynomial differential equation \eqref{ELmatrix}; hence the corresponding flat output is smooth, and so are $x = \cX(D) y$ and $u = \cU(D) y$. For such extremals, endpoint traces of $u$ and of its derivatives are perfectly well-defined and become ordinary linear boundary conditions on a finite jet of $y$, or equivalently, after passage to a first-order realization of \eqref{ELmatrix}, on the corresponding endpoint state. If these extra conditions are combined with the natural transversality relations and if the resulting two-point boundary operator is admissible, then the same hyperbolic argument yields a unique classical extremal and the same exponential turnpike estimate.

\medskip
This observation is closely related to the standard dynamic-extension trick. For instance, fixing $u(0)$ and $u(T)$ can be rewritten by adjoining $u$ to the state and taking $\dot u = v$ as the new control. The new running cost then penalizes $v$ with a semidefinite control weight, often actually vanishing in the added directions. In that sense, endpoint constraints on the control are another manifestation of the semidefinite regimes highlighted in the present paper.
\end{remark}

\section{Hyperbolic two-point boundary value problems}
\label{hypsec}
We now isolate the elementary hyperbolic estimate that drives the turnpike proof.

\begin{definition}
\label{admissibledef}
Let $\mathbb{A} \in \R^{N \times N}$ have no eigenvalue on the imaginary axis. Denote by $E^s$ and $E^u$ its stable and unstable spectral subspaces. Given matrices $\mathbb{C}_0,\mathbb{C}_1 \in \R^{N \times N}$, we say that the boundary operator
\begin{equation}
\label{admissible-op}
\mathbb{C}_0 Z(0) + \mathbb{C}_1 Z(T) = d
\end{equation}
is \emph{admissible} if the map $\cB_\infty : E^s \times E^u \to \R^N$ defined by
\begin{equation*}
\cB_\infty(a,b) = \mathbb{C}_0 a + \mathbb{C}_1 b
\end{equation*}
is an isomorphism.
\end{definition}

This simply means that the left-boundary condition sees the stable modes and the right-boundary condition sees the unstable modes with full rank.

\begin{theorem}
\label{abstract-bvp}
Let $\mathbb{A} \in \R^{N \times N}$ have no eigenvalue on the imaginary axis, and let $E^s \oplus E^u$ be its stable-unstable splitting. Assume that the boundary operator \eqref{admissible-op} is admissible in the sense of Definition~\ref{admissibledef}. Then there exist $T_0 > 0$ and $C,\mu > 0$ such that, for every $T \geq T_0$ and every $d \in \R^N$, the two-point boundary problem
\begin{equation*}
\dot Z(t) = \mathbb{A} Z(t), \qquad \mathbb{C}_0 Z(0) + \mathbb{C}_1 Z(T) = d,
\end{equation*}
has a unique solution and
\begin{equation}
\label{abstract-est}
\Vert Z(t) \Vert \leq C \Big( e^{-\mu t} + e^{-\mu (T - t)} \Big) \Vert d \Vert \qquad \forall t \in [0,T].
\end{equation}
\end{theorem}

\begin{proof}
Since $\mathbb{A}$ is hyperbolic, there exist spectral projectors $\Pi^s$, $\Pi^u$ onto $E^s$, $E^u$, and constants $c,\mu > 0$ such that
\begin{equation}
\label{semigroup-est}
\big\Vert e^{t \mathbb{A}} \Pi^s \big\Vert \leq c e^{-\mu t}, \qquad \big\Vert e^{-t \mathbb{A}} \Pi^u \big\Vert \leq c e^{-\mu t} \qquad
\forall t \geq 0.
\end{equation}
For every solution of $\dot Z = \mathbb{A} Z$, define $a = \Pi^s Z(0) \in E^s$ and $b = \Pi^u Z(T) \in E^u$.
Then
\begin{equation}
\label{decomposition}
Z(t) = e^{t \mathbb{A}} a + e^{(t - T) \mathbb{A}} b.
\end{equation}
Indeed, the first term carries the stable part from the left endpoint and the second term carries the unstable part backward from the right endpoint.

Plugging \eqref{decomposition} into the boundary condition gives $\cB_T(a,b) = d$, where
\begin{equation*}
\cB_T(a,b) = \mathbb{C}_0 \big( a + e^{-T \mathbb{A}} b \big) + \mathbb{C}_1 \big( e^{T \mathbb{A}} a + b \big).
\end{equation*}
By \eqref{semigroup-est},
$$
\big\Vert e^{T \mathbb{A}} a \big\Vert \leq c e^{-\mu T} \Vert a \Vert \qquad \forall a \in E^s,
$$
and similarly
$$
\big\Vert e^{-T \mathbb{A}} b \big\Vert \leq c e^{-\mu T} \Vert b \Vert \qquad \forall b \in E^u.
$$
Hence $\cB_T = \cB_\infty + \mathrm{O} \big( e^{-\mu T} \big)$ as operators on $E^s \times E^u$. Since $\cB_\infty$ is invertible by admissibility, $\cB_T$ is invertible for every $T$ large enough, with uniformly bounded inverse. This yields
\begin{equation}
\label{ab-est}
\Vert a \Vert + \Vert b \Vert \leq C_1 \Vert d \Vert.
\end{equation}
Finally, using \eqref{semigroup-est}, \eqref{decomposition} and \eqref{ab-est}, one obtains
$$
\Vert Z(t) \Vert \leq c e^{-\mu t} \Vert a \Vert + c e^{-\mu (T - t)} \Vert b \Vert \leq C \Big( e^{-\mu t} + e^{-\mu (T - t)} \Big) \Vert d \Vert,
$$
which is \eqref{abstract-est}.
\end{proof}

\begin{remark}
\label{scalarremark}
For a scalar equation $d(D) z = 0$ with no purely imaginary root, admissibility is completely explicit: one writes the general solution as a sum of stable modes $e^{\lambda t}$, $\Rea \lambda < 0$, and unstable modes $e^{\lambda (t - T)}$, $\Rea \lambda > 0$, and one checks that the chosen boundary functionals form an invertible square matrix when evaluated on those modes. Proposition~\ref{doubleprop2} is exactly this criterion in the simplest second-order situation.
\end{remark}

\section{The flatness-based exponential turnpike theorem}\label{mainsec}
We now return to the linear-quadratic problem \eqref{mainprob}. Let $\cE(D)$ be the reduced Euler-Lagrange operator given by \eqref{ELoperator}, and let $N$ be defined by \eqref{Ndef}. By Proposition~\ref{smithprop}, there exists a first-order realization of \eqref{ELmatrix},
\begin{equation}
\label{realization}
\dot Z(t) = \mathbb{A} Z(t),
\qquad
y(t) = \mathbb{L} Z(t),
\end{equation}
with $Z(t) \in \R^N$, such that the spectrum of $\mathbb{A}$ consists exactly of the characteristic roots of $\cE(D)$, counted with algebraic multiplicity. The prescribed endpoint condition \eqref{flat-bc} together with the natural transversality conditions of Proposition~\ref{ELprop} becomes a linear two-point boundary condition
\begin{equation}
\label{realbc}
\mathbb{C}_0 Z(0) + \mathbb{C}_1 Z(T) = \eta,
\end{equation}
for some matrices $\mathbb{C}_0,\mathbb{C}_1 \in \R^{N \times N}$ and some vector $\eta \in \R^N$ determined by the endpoint data and by the affine terms of the original problem, if any.

\begin{theorem}
\label{mainth}
Consider the linear-quadratic problem \eqref{mainprob}. Assume that:
\begin{enumerate}[label={\rm (H\arabic*)},leftmargin=2.4em]
\item \label{H1} the pair $(A,B)$ is controllable;
\item \label{H2} the reduced Euler-Lagrange operator $\cE(D)$ is hyperbolic, namely \eqref{hyper-cond} holds;
\item \label{H3} the boundary operator \eqref{realbc} is admissible in the sense of Definition~\ref{admissibledef}.
\end{enumerate}
Then there exist $T_0 > 0$ and $C,\mu > 0$ such that, for every $T \geq T_0$, the problem \eqref{mainprob} has a unique optimal pair $(x^T,u^T)$ and
\begin{equation}
\label{turnpike-est}
\vert x^T(t) \vert + \vert u^T(t) \vert \leq C \Big( e^{-\mu t} + e^{-\mu (T - t)} \Big)
\qquad
\forall t \in [0,T].
\end{equation}
More generally, the same estimate holds for every derivative of the flat output that enters the expressions of $x$ and $u$.
\end{theorem}

\begin{proof}
By Proposition~\ref{flatprop}, the problem reduces to the higher-order variational problem \eqref{flat-cost}-\eqref{flat-bc}. Because the functional is convex, every stationary point is a global minimizer. By Proposition~\ref{ELprop}, stationary points are characterized by the reduced equation \eqref{ELmatrix} together with the two-point boundary condition \eqref{realbc}. By \ref{H2}, the realization matrix $\mathbb{A}$ in \eqref{realization} has no eigenvalue on the imaginary axis. By \ref{H3}, the boundary operator is admissible. Therefore Theorem~\ref{abstract-bvp} applies and yields a unique solution $Z^T$ of \eqref{realization}-\eqref{realbc}, together with the estimate
$\Vert Z^T(t) \Vert \leq C_1 \big( e^{-\mu t} + e^{-\mu (T - t)} \big)$.
Since $x = \cX(D) y$ and $u = \cU(D) y$, and since $y = \mathbb{L} Z$, both $x$ and $u$ are fixed linear combinations of the components of $Z$ and hence satisfy \eqref{turnpike-est}. Uniqueness of the stationary point implies uniqueness of the minimizer.
\end{proof}

\begin{remark}
If the original problem is centered around a nonzero static optimizer $(\bar x,\bar u)$, then Theorem~\ref{mainth} gives
$$
\vert x^T(t) - \bar x \vert + \vert u^T(t) - \bar u \vert \leq C \Big( e^{-\mu t} + e^{-\mu (T - t)} \Big).
$$
Thus the theorem provides the expected exponential turnpike around the static optimum.
\end{remark}

\begin{remark}
Assumption \ref{H3} is the precise formulation of the informal statement that the optimal solution must be \emph{sufficiently constrained at both ends}. The hypothesis is not restrictive in the regular case $R > 0$ with fixed initial and final states: then $N = 2n$, and prescribing the state at $0$ and at $T$ exactly matches the dimension of the stable and unstable bundles. In semidefinite cases, however, the effective order $N$ may be smaller than $2n$, and admissibility becomes a genuine issue. This is the mechanism behind Proposition~\ref{doubleprop2}.
\end{remark}

\begin{remark}
In the regular case $R > 0$ with prescribed initial and final states, Assumption \ref{H3} is not an additional mysterious requirement. Once the stationary solution is written as a sum of left-stable and right-unstable modes, the initial state fixes the stable amplitudes and the final state fixes the unstable amplitudes. Thus admissibility is exactly the nondegeneracy built into the classical shooting argument for the hyperbolic Hamiltonian system (see \cite[end of Section 3.2]{TrelatZuazua2015}).
\end{remark}

\begin{remark}
Hyperbolicity is the natural generic condition. If $\cE(D)$ has a purely imaginary root, then Proposition~\ref{smithprop} produces oscillatory or polynomial modes. Unless the endpoint conditions annihilate those modes by a nongeneric compatibility relation, such components remain visible in the interior and prevent an exponential turnpike estimate. This explains why the loss of detectability in the double-integrator example leads to the linear turnpike of \cite{Trelat2023} rather than to an exponential one.
\end{remark}

\begin{corollary}
Assume that $R > 0$ and that $(A,Q^{\frac{1}{2}})$ is detectable. Then the classical Hamiltonian matrix associated with \eqref{mainprob} has no purely imaginary eigenvalue. Since its spectrum coincides with the characteristic roots of the reduced Euler-Lagrange equation, Assumption \ref{H2} holds. Therefore Theorem~\ref{mainth} recovers the classical exponential turnpike theorem for regular finite-dimensional linear-quadratic problems, under the admissibility of the endpoint conditions.
\end{corollary}

\begin{proof}
The Hamiltonian hyperbolicity under $R > 0$ and detectability is classical. Eliminating the state and the control in flat coordinates yields the same characteristic equation. The conclusion follows from Theorem~\ref{mainth}.
\end{proof}

\section{Conclusion and perspectives}

The main point of this article is that flatness yields a very direct explanation of the exponential turnpike phenomenon for controllable linear-quadratic problems. After parametrization by a flat output, the optimality system becomes a constant-coefficient higher-order boundary value problem. The turnpike mechanism is then reduced to two elementary facts: the reduced Euler-Lagrange operator must be hyperbolic, and the endpoint conditions must act nondegenerately on the stable and unstable modes. In the regular case $R > 0$, this recovers in flat coordinates the classical Hamiltonian-Riccati picture. When $R$ is only semidefinite, the same reduction makes visible the possible order drop of the reduced equation, the ensuing loss of admissible boundary conditions, and the appearance of polynomial or oscillatory modes. This is precisely where the flat viewpoint seems to add new information with respect to the standard regular theory. The same framework also provides a natural interpretation of endpoint constraints on the control and on finitely many derivatives of the control: these traces are not defined on the ambient $L^2$ class, but they become ordinary boundary conditions on the smooth flat extremals.

\medskip

Several questions remain open. 
A first one is to understand how far this strategy extends beyond the linear-quadratic setting. For nonlinear flat systems (see \cite{batna} for first hints), one may hope to reduce the optimality system to nonlinear differential equations in flat outputs and to recover turnpike information from a suitable hyperbolic analysis near a steady solution. 
A second question is to clarify the correct functional framework in semidefinite regimes when classical solutions fail to satisfy all endpoint conditions and one has to pass to impulsive, measure-valued, or relaxed controls. One may even hope that nonstandard analysis could provide a simple and appropriate setting for such singular limits, by representing boundary layers together with concentration and oscillation phenomena within a single infinitesimal formalism (see \cite{Nelson1977,Neves2004}).\footnote{See, e.g., \cite{lobry} for an intuitive presentation of nonstandard analysis and some applications in control theory.}
A third direction is to characterize hyperbolicity and admissibility directly in intrinsic flat terms, without passing through a particular Smith reduction or realization. Finally, it would be natural to investigate periodic turnpikes, manifold turnpikes, and descriptor or distributed-parameter analogues from the same viewpoint.


\small

\begin{thebibliography}{99}

\bibitem{AskovicTrelatZidani2024}
V. A\v{s}kovi\'c, E. Tr\'elat, H. Zidani, 
\textit{Linear quadratic optimal control turnpike in finite and infinite dimension: two-term expansion of the value function}, 
Systems Control Lett. {\bf 188} (2024), 105803.

\bibitem{Bliss1946}
G. A. Bliss, 
\textit{Lectures on the Calculus of Variations}, 
University of Chicago Press, Chicago, 1946.

\bibitem{Brunovsky1970}
P. Brunovsk\'y, 
\textit{A classification of linear controllable systems}, 
Kybernetika {\bf 6} (1970), no. 3, 173--188.

\bibitem{CourantHilbert1953}
R. Courant, D. Hilbert, 
\textit{Methods of Mathematical Physics. Vol. I}, 
Interscience Publishers, New York, 1953.

\bibitem{EsteveKouhkouhPighinZuazua2022}
C. Esteve, H. Kouhkouh, D. Pighin, E. Zuazua, 
\textit{The turnpike property and the longtime behavior of the Hamilton-Jacobi-Bellman equation for finite-dimensional LQ control problems}, 
Math. Control Signals Systems {\bf 34} (2022), 819--853.

\bibitem{Fliess1990}
M. Fliess, 
\textit{Some basic structural properties of generalized linear systems}, 
Systems Control Lett. {\bf 15} (1990), no. 5, 391--396.

\bibitem{FliessLevineMartinRouchon1995}
M. Fliess, J. L\'evine, P. Martin, P. Rouchon, 
\textit{Flatness and defect of non-linear systems: introductory theory and examples}, 
Internat. J. Control {\bf 61} (1995), no. 6, 1327--1361.

\bibitem{GelfandFomin1963}
I. M. Gelfand, S. V. Fomin, 
\textit{Calculus of Variations}, 
Revised English edition, translated and edited from the Russian by Richard A. Silverman, Prentice-Hall, Englewood Cliffs, NJ, 1963.

\bibitem{Goh1966}
B. S. Goh, 
\textit{The second variation for the singular Bolza problem}, 
SIAM J. Control {\bf 4} (1966), no. 2, 309--325.

\bibitem{GohbergLancasterRodman1982}
I. Gohberg, P. Lancaster, L. Rodman, 
\textit{Matrix Polynomials}, 
Academic Press, New York, 1982.

\bibitem{GruneGuglielmi2021}
L. Gr\"une, R. Guglielmi, 
\textit{On the relation between turnpike properties and dissipativity for continuous-time linear-quadratic optimal control problems}, 
Math. Control Relat. Fields {\bf 11} (2021), no. 1, 169--188.

\bibitem{GuglielmiLi2024}
R. Guglielmi, Z. Li, 
\textit{Necessary conditions for turnpike property for generalized linear-quadratic problems}, 
Math. Control Signals Systems {\bf 36} (2024), 799--829.

\bibitem{Hestenes1966}
M. R. Hestenes, 
\textit{Calculus of Variations and Optimal Control Theory}, 
John Wiley \& Sons, New York, 1966.

\bibitem{JoinDelaleauFliess2024}
C. Join, E. Delaleau, M. Fliess, 
\textit{Flatness-based control revisited: the HEOL setting}, 
C. R. Math. Acad. Sci. Paris {\bf 362} (2024), 1693--1706.

\bibitem{batna}
C. Join, E. Delaleau, M. Fliess,  
\textit{The Euler-Lagrange equation and optimal control: Preliminary results},
12th Internat. Conf. Syst. Contr., Batna (Algeria) (2024), 155-160, IEEE Xplore.

\bibitem{JoinDelaleauFliess2025}
C. Join, E. Delaleau, M. Fliess, 
\textit{Linear Quadratic Regulators: A New Look}, 
arXiv:2512.10641, 2025.

\bibitem{Kailath1980}
T. Kailath, 
\textit{Linear Systems}, 
Prentice-Hall, Englewood Cliffs, NJ, 1980.

\bibitem{KokotovicKhalilOReilly1986}
P. V. Kokotovi\'c, H. K. Khalil, J. O'Reilly, 
\textit{Singular Perturbation Methods in Control: Analysis and Design}, 
Academic Press, London, 1986.

\bibitem{KwakernaakSivan1972}
H. Kwakernaak, R. Sivan, 
\textit{The maximally achievable accuracy of linear optimal regulators and linear optimal filters}, 
IEEE Trans. Automat. Control {\bf 17} (1972), no. 1, 79--86.

\bibitem{Levine2009}
J. L\'evine, 
\textit{Analysis and Control of Nonlinear Systems: A Flatness-Based Approach}, 
Springer, Berlin, 2009.


\bibitem{lobry}
C. Lobry, T. Sari, 
{\it Non-standard analysis and representation of reality} Int. J. Contr. {\bf 81} (2008), 519-536.

\bibitem{Nelson1977}
E. Nelson,
\textit{Internal set theory: A new approach to nonstandard analysis},
Bull. Amer. Math. Soc. {\bf 83} (1977), no. 6, 1165--1198.

\bibitem{Neves2004}
V. Neves,
\textit{Nonstandard calculus of variations},
J. Math. Sci. {\bf 120} (2004), no. 1, 940--954.

\bibitem{OReilly1983}
J. O'Reilly, 
\textit{Partial cheap control of the time-invariant regulator}, 
Internat. J. Control {\bf 37} (1983), no. 5, 909--927.

\bibitem{SaberiSannuti1987}
A. Saberi, P. Sannuti, 
\textit{Cheap and singular controls for linear quadratic regulators}, 
IEEE Trans. Automat. Control {\bf 32} (1987), no. 3, 208--219.

\bibitem{Schwartz1965}
L. Schwartz, 
\textit{M\'ethodes math\'ematiques pour les sciences physiques}, 
Hermann, Paris, 1965.

\bibitem{Trelat2023}
E. Tr\'elat, 
\textit{Linear turnpike theorem}, 
Math. Control Signals Systems {\bf 35} (2023), no. 3, 685--739.

\bibitem{TrelatZhangZuazua2018}
E. Tr\'elat, C. Zhang, E. Zuazua, 
\textit{Steady-state and periodic exponential turnpike property for optimal control problems in Hilbert spaces}, 
SIAM J. Control Optim. {\bf 56} (2018), no. 2, 1222--1252.

\bibitem{TrelatZuazua2015}
E. Tr\'elat, E. Zuazua, 
\textit{The turnpike property in finite-dimensional nonlinear optimal control}, 
J. Differential Eq. {\bf 258} (2015), no. 1, 81--114.

\bibitem{TrelatZuazua2025}
E. Tr\'elat, E. Zuazua, 
\textit{Turnpike in optimal control and beyond: a survey}, 
arXiv:2503.20342, 2025.

\bibitem{Vinter2000}
R. Vinter, 
\textit{Optimal Control}, 
Birkh\"auser, Boston, 2000.

\bibitem{VinterLewis1978}
R. B. Vinter, R. M. Lewis, 
\textit{The equivalence of strong and weak formulations for certain problems in optimal control}, 
SIAM J. Control Optim. {\bf 16} (1978), 546--570.

\bibitem{Yosida1984}
K. Yosida, 
\textit{Operational Calculus: A Theory of Hyperfunctions}, 
Springer, New York, 1984.

\end{thebibliography}
\end{document}